\documentclass[12pt]{article}
\usepackage{amssymb}
\usepackage{amsthm}
\usepackage{bm}
\begin{document}
	\newtheorem{fact}{Fact}
\newtheorem{thm}{Theorem}
\newtheorem{conv}{Convention}
\newtheorem{cor}{Corollary}
\newtheorem{lem}{Lemma}
\newtheorem{slem}{Sublemma}
\newtheorem{prop}{Proposition}
\newtheorem{defn}{Definition}
\newtheorem{conj}{Conjecture}
\newtheorem{exple}{Exemple}
\newtheorem{ass}{Assumption}
\newtheorem{ques}{Question}
\title{Minimal surfaces in $\mathbb{R}^4$ and their links at infinity}
\author{Marc Soret and Marina Ville}
\date{ }
\maketitle
\begin{abstract}
We look at complete embedded/immersed minimal surfaces of finite total curvature in $\mathbb{R}^4$. Similarly to the case of complex curves in $\mathbb{C}^2$ we introduce their link at infinity and we define a writhe number at infinity which gives us a formula for the total normal curvature of the surface. The knowledge of the link at infinity can sometimes  help us determine if a surface has self-intersection and we illustrate this idea by looking surfaces of small total curvature. In particular we classify embedded minimal surfaces of total curvature $-4\pi$.
\end{abstract}
\section{Introduction - Sketch of the paper}
\let\thefootnote\relax\footnote{MSC 2020 - Mathematical Sciences Classification System 53C42, 57K10}
We consider here complete embedded/immersed minimal surfaces of finite total curvature in $\mathbb{R}^4$. In a seminal paper, Chern and Osserman ([C-O]) showed that the Gauss maps giving the data of their oriented tangent planes can be seen as a pair of two meromorphic functions. \\
 In $\mathbb{R}^3$, a complete surface with a constant Gauss map is a plane. In $\mathbb{R}^4$, if a complete surface has both Gauss maps constant, it is a plane; if one of them is constant, it is a complex curves w.r.t. to some orthogonal complex structure in $\mathbb{R}^4$.\\
Much research has been done about the Gauss maps, using tools of complex analysis as it gives us good information about the minimal surface. 
However it cannot really help us determined when an immersed minimal surface is actually embedded and we  address this problem here.\\
\\
We start by recalling various definitions of the Gauss maps and explain how to go from one to the other.\\
From then on, we assume that our minimal surface $S$  is properly immersed and of finite total curvature. If $S$ is embedded near infinity, we define its {\it link at infinity}, which is the intersection of the surface with a sphere of very large radius in $\mathbb{R}^4$. We give a formula relating this link to the total normal curvature of the surface and derive some restrictions on the asymptotic behaviour of the surface. The   knot type of a generic end is a torus knot and we investigate the generic non torus knot possiblities. \\
Finally we look at minimal surfaces of small total curvature. If the curvature is $-4\pi$, we  classify all complete embedded non holomorphic ones. We get some partial information for curvatures $-6\pi$ and $-8\pi$. This being said, we are left with the basic
\begin{ques}
	What are the minimal embeddings of $\mathbb{C}$ into $\mathbb{R}^4$? Does such an embedding exist for every total curvature $-2n\pi$, with $n$ a positive integer? Or for every total curvature $-4n\pi$?
	\end{ques}	
{\bf Acknowledgments} We thank Mario Micallef for directing us to [H-O1] and Hojoo Lee for noticing the presence in this paper of the Enneper surface and the catenoid.
\section{Classical formalism}
\subsection{Definition - Local parametrization}\label{local expression for a minimal surface}
A minimal map $F$ from a surface to $\mathbb{R}^4$ is conformal and harmonic.  If we identify $\mathbb{R}^4$ with 
$\mathbb{C}^2$, we can locally parametrize $F$ by maps 
\begin{equation}\label{condition harmonique}
z\in\mathbb{D}\mapsto (e(z)+\overline{f}(z),g(z)+\overline{h}(z))
\end{equation}
where $\mathbb{D}$ is the unit disk of $\mathbb{C}$ and $e,f,g,h$ are holomorphic functions verifying
\begin{equation}\label{condition conforme}
e'f'+g'h'=0
\end{equation}
Condition (\ref{condition conforme}) ensures that $F$ is conformal. The representation of $F$ by $e,f,g,h$ can be seen as a 
 $4$-dimensional analogue of the Weierstrass representation in $\mathbb{R}^3$.
\subsection{The Gauss maps}
If $\Sigma$ is an oriented surface in $\mathbb{R}^4$, there are two Gauss maps $\gamma_{\pm}:\Sigma\longrightarrow\mathbb{R}^4$.  We recall three equivalent definitions for $\gamma_+$ (the case of $\gamma_-$ is similar) as each definition has different advantages.\\
Let $p\in\Sigma$ and let $(u,v)$ a positive orthonormal basis of  the tangent plane $T_p\Sigma$.
\subsubsection{The Eells-Salamon approach ([E-S])}\label{eells-salamon subsection}
Eells-Salamon identify the set of complex structures on $\mathbb{R}^4$ which preserve the metric and preserve (resp. reverse) the orientation of $\mathbb{R}^4$ with the $2$-sphere $Z_+$ (resp. $Z_-$) of unit $2$-vectors $\alpha$ with $\star \alpha=\alpha$ (resp. $\star \alpha=-\alpha$); here $\star:\Lambda^2(\mathbb{R}^4)\longrightarrow \Lambda^2(\mathbb{R}^4)$ is the Hodge operator.  We set $\gamma_\pm(p)$ to be the unique such complex structure $J$ such that $T_p\Sigma$ is an oriented complex $J$-line. In other words, $J(u)=v$.\\
Note that this definition does not depend on the choice of a basis of $\mathbb{R}^4$, which enabled [E-S] to extend it to a general Riemannian manifold and create the twistor theory of minimal surfaces.\\
\\ By contrast, the next two definitions depend on the choice of a positive orthonormal basis $(e_1,e_2,e_3,e_4)$ of $\mathbb{R}^4$.
\subsubsection{The quaternions approach ([A-S-V])}\label{definition par les quaternions}
We identify $\mathbb{R}^4$ with the quaternions $\mathbb{H}$
by setting
\[x_1e_1+x_2e_2+x_3e_3+x_4e_4\mapsto x_1+x_2i+x_3j+x_4k\] and we define 
\begin{equation}
\label{quaternions}
\gamma_+(p)=v\cdot u^{-1}
\end{equation}
It is a unit imaginary quaternion so it is of the form  $\gamma_+(p)=ai+bj+ck$ and we identify it with the complex structure $J$ in $Z_+$ such that $Je_1=ae_2+be_3+ce_4$. 
\subsubsection{The classical approach ([C-O], [M-O], [H-O1], [H-O2] etc)}\label{classique} 
We complexify $\mathbb{R}^4$, write $u-iv=(w_1,w_2,w_3,w_4)\in 
\mathbb{R}^4\otimes\mathbb{C}$ and set
\begin{equation}
\label{la forme de osserman et al}
\gamma_{+}(p)=\frac{w_3+iw_4}{w_1-iw_2}\in\mathbb{C}P^1
\end{equation}
We can see (\ref{la forme de osserman et al}) as a stereographic projection of (\ref{quaternions}):
\begin{lem}\label{lemme pour l'identification}
	If $u\cdot v^{-1}=ai+bj+ck$ (\S \ref{definition par les quaternions}), the $w_i$'s defined in \S \ref{classique} verify
	\begin{equation}\label{egalite entre les deux gauss maps1}
	\frac{w_3+iw_4}{-w_1+iw_2}=-i\frac{b+ic}{1-a}
	\end{equation}
\end{lem}
\begin{proof}
	Fix $J$ by setting $J(e_1)=ae_2+be_3+ce_4$. A computation shows that the $(w_1,w_2,w_3,w_4)$'s given by $e_k-iJe_k$ for $i=1,2,3,4$ all verify (\ref{egalite entre les deux gauss maps1}). Hence (\ref{egalite entre les deux gauss maps1}) is true for $u-iJu=u-iv$.
	\end{proof}
If we put together these various approaches, we derive
\begin{exple}
If $e,f$ are holomorphic functions and $\lambda\in\mathbb{C}$, then
	\begin{equation}
	(e+\overline{f},\lambda f-\frac{\overline{e}}{\lambda})
	\end{equation}
	is minimal and holomorphic w.r.t. some orientation preserving orthogonal complex structure on $\mathbb{R}^4$.
	\end{exple}

\subsection{Complete minimal surfaces of finite total curvature}\label{generalite sur les courbures}
From now on we assume:
\begin{ass}\label{definition du pmftc}
   $\Sigma$ is an oriented complete minimal surface of finite total curvature  properly immersed in $\mathbb{R}^4$ (PMFTC).
\end{ass} We recall some classical properties ([C-O], [J-M] etc).
\begin{itemize}
	\item 
There exists a compact Riemann surface $\hat{\Sigma}$ without boundary and a finite number of points $p_1,...,p_d$ in 
$\hat{\Sigma}$ such that 
$\Sigma=\hat{\Sigma}\backslash\{p_1,...,p_d\}$
and the Gauss maps $\gamma_{\pm}$ extends to holomorphic maps $\hat{\gamma}_\pm:\hat{\Sigma}\longrightarrow \mathbb{C}P^1.$ We denote by $d_+$ (resp. $d_-$) the degree of $\hat{\gamma}_+$ (resp. $\hat{\gamma}_-$). If $K^T$ and $K^N$ are the curvatures of the tangent and normal bundles of $\Sigma$
\begin{equation}\label{equation:the two curvatures}
-\int_{\Sigma}K^T=2\pi (d_++d_-)\ \ \ \ \ -\int_{\Sigma}K^N=2\pi (d_+-d_-)
\end{equation}
REMARK. Depending on the conventions, the Gauss maps are holomorphic or antiholomorphic. Here we take them to be holomorphic so $d_\pm\geq 0$. If one of the $d_\pm$ is zero, the surface is complex for some orthogonal complex structure on $\mathbb{R}^4$.
\item 
For $R$ large enough, $\Sigma\cap(\mathbb{R}^4\backslash \mathbb{S}(0,R))$ is a finite union of annuli, called  {\it ends} of $\Sigma$, each end corresponding to a $p_k$.\\
 Each end $E$ has a parametrization as  
\[\{z/|z|>R\}\longrightarrow\mathbb{R}^4=\mathbb{C}\times\mathbb{C}\]
\begin{equation}\label{parametrisation d'un bout par un anneau}
z\mapsto(z^N+ o(|z^N|), o(|z^N|))
\end{equation}
The plane defined by the first complex coordinate in (\ref{parametrisation d'un bout par un anneau}) is the {\it tangent plane at infinity} $T_E\Sigma$ of the end; its orthogonal complement $N_E\Sigma$ is the {\it normal plane at infinity}.
The quantity $N-1$ is called the {\it branching order at infinity} for the end and $N$ is the {\it order of the end}.

\end{itemize}

\section{The link/knot/braid at infinity}
In this section we assume that $\Sigma$ is a PMFTC (see Assumption \ref{definition du pmftc}) embedded outside a ball in $\mathbb{R}^4$. We define a link at infinity similarly to the case of complex curves (cf. [N-R]).
\subsection{Construction}
\subsubsection{Reminder: braids}
A {\it closed braid} in $\mathbb{R}^3$ with oriented axis  $Oz$ is a loop $\gamma(t)$ whose cylindrical coordinates $(\rho(t), \theta(t), z(t))$ verify for all $t$,
\[\rho(t)\neq 0, \ \ \ \ \theta'(t)>0\]
 The number of strands of $\gamma$ is the degree of $\theta:\mathbb{S}^1\longrightarrow :\mathbb{S}^1$. The algebraic length $e(\gamma)$ is the linking number of $\gamma$ with a loop $\hat{\gamma}$ obtained by pushing $\gamma$ slightly in the direction of $Oz$.\\
 If we add a point at infinity to $\mathbb{R}^3$, we get a braid in $\mathbb{S}^3$ with a great circle for its axis.\\
 We can also write a braid with $N$ strands as an element of the braid group $B_N$ generated by the $\sigma_i$'s which exchange the $i$-th strand with the $i+1$-th strand. For a braid $\prod_{i} \sigma_{k_i}^{m(i)}$, we have
 \begin{equation}
 e(\prod_{i} \sigma_{k_i}^{m(i)})=\sum_{i}m(i)
 \end{equation}
\subsubsection{For a single end}
Consider an end $E$ of the surface $\Sigma$ parametrized as in (\ref{parametrisation d'un bout par un anneau}). Clearly, there exists an $R_0$ such that, if $R>R_0$, $\Sigma$ is transverse to $\mathbb{S}(0,R)$; thus \[K_R=\mathbb{S}(0,R)\cap\Sigma\] is a knot which does not depend on $R>R_0$, up to isotopy. We take a non zero vector $X_E$ in $N_E\Sigma$. We project $X_E$ on $\mathbb{S}(0,R)$, push $K_R$ in that direction and get another knot $\hat{K}_R$. The linking number
$lk(K_R,\hat{K}_R)$ is the {\it self-linking number} of $K_R$ w.r.t. the framing $X_E$. 
\begin{defn}
	For $R$ large enough, the {\it self-linking number} of $K_R$ w.r.t. the framing $X_E$ does not depend on the vector $X_E$ orthogonal to the plane tangent at infinity at the end $E$. It also does not depend on $R$ and 
we call it the {\bf writhe at infinity} $w_\infty(E)$ of the end $E$.
\end{defn}
\subsubsection{The braid at infinity}
If $N$ is the order at infinity of the end, $K_R$ can be seen as an $N$-braid in the $3$-sphere with axis the unit circle in $N_E\Sigma$. Equivalently, we can view it as a braid in a cylinder $C(R)$:
\[K_R=\Sigma\cap\underbrace{\{(z_1,z_2)\in\mathbb{C}^2\slash |z_1|=R\}}_{C(R)}\]
with $(z_1,z_2)$ complex coordinates in $\mathbb{R}^4\cong\mathbb{C}^2$, $z_1$ generates $T_E\Sigma$ and $z_2$ generates $N_E\Sigma$.\\
We can take $X_E$ as a generator of the axis of the braid. The writhe $w_\infty(E)$ is the algebraic length of the braid.
\subsubsection{Several ends}\label{several ends}
If $\Sigma$ has several ends, the union of the knots of the various ends is the {\it link at infinity} of $\Sigma$ which we denote $L_\infty(\Sigma)$.\\
We define a framing $X$ on $\Sigma\cap\mathbb{S}(0,R)$ by putting together the $X_{E_i}$'s corresponding to the various ends $E_i$. The {\it writhe at infinity of $\Sigma$}, denoted $w_\infty(\Sigma)$ is the self-linking number of the link $\Sigma\cap\mathbb{S}(0,R)$ w.r.t. the framing $X$.
\begin{prop} If $\Sigma$ has $k$ mutually transverse tangent planes $P_i$ at infinity, its writhe at infinity of $\Sigma$ is 
 \begin{equation}
w_\infty(\Sigma)=\sum_{i=1}^k w_\infty(E_i)+\sum_{1\leq i<j\leq k}N_iN_j\sigma(i,j)
 \end{equation}
 \begin{minipage}{0.1\textwidth}
 where
 \end{minipage}
\begin{minipage}{0.9\textwidth}
\begin{itemize}
 	\item 
 each end has a branching order $N_i-1$
 \item
 $\sigma(i,j)$ is $1$ (resp. $-1$) if $P_i$ and $P_j$ intersect positively (resp. negatively).
 \end{itemize}
\end{minipage}
 \end{prop}
\subsection{Integral formulae for the  curvatures}\label{integral formulae for the normal bundle}
The following theorem is classical.
\begin{thm}\label{thm: integrale de la courbure tangente}([Shi]) If  $\Sigma$ has $k$ ends of branching order $N_i-1$, $i=1,...,k$ (cf. (\ref{parametrisation d'un bout par un anneau})) and Euler characteristic $\chi(\Sigma)$, then
\begin{equation}
\frac{1}{2\pi}\int_{\Sigma}K^T=-\sum_{i=1}^k N_i+\chi(\Sigma)
\end{equation}
\end{thm}
\begin{thm}\label{proposition:curvature of the normal bundle}  
	If $\Sigma$ has $D_\Sigma$ transverse double points counted with sign,
		\begin{equation}\label{courbure du fibre normal;immerge}
		\frac{1}{2\pi}\int_{\Sigma}K^N=w_\infty(\Sigma) -2D_\Sigma
		\end{equation}

\end{thm}
\begin{proof} 
The vector fields $X_E$ defined above are very close to belonging to $N\Sigma$ so we define a section $X^N$ of $N\Sigma$ which is very close to $X_E$ on each end $E$.  We push $\Sigma\cap\mathbb{B}(0,R) $ in the direction of $X^N$ to get a surface $\hat{\Sigma}$ in $\mathbb{B}(0,R)$ bounded by $\hat{K}_R$. The linking number of two links $L_1$ and $L_2$ is the number of intersection points of two surfaces, each of them bounded by one of the $L_i$'s; hence $w_\infty(E)$ is the number of intersection points between $\Sigma\cap\mathbb{B}(0,R) $ and  $\hat{\Sigma}\cap\mathbb{B}(0,R)$. \\
If we denote by $Z(X^N)$ the number of zeroes of $X^N$, this means
\begin{lem}\label{nombre de zeros}  If $\Sigma$ is immersed with $D_\Sigma$ transverse double points,
	$w_\infty(\Sigma)=Z(X^N)-2D_\Sigma$.	
\end{lem}
We now investigate $Z(X^N)$.
	 We let $J_+$ be the complex structure (compatible with the metric and orientation) on $N\Sigma$ and we apply Stokes'theorem to the form $\omega$ on $N\Sigma$ \[\omega=-\frac{1}{\|X^N\|^2}<\nabla X^N, JX^N>.\] 
	
Since $d\omega=K^NdA$, where $dA$ is the area element on $\Sigma$, we have
\[\frac{1}{2\pi}\int_{\Sigma\cap B(0,R)}K^N+\frac{1}{2\pi}\int_{\Sigma\cap \partial S(0,R)}\omega=Z(X^N).\]
 As $R$ becomes large, $X^N$ is asymptotic to the constant vector $X_E$ on each end, thus
	\[\lim_{R\longrightarrow\infty}\int_{E\cap \partial S(0,R)}\omega=0\]
\end{proof} 
REMARK. We have similar contructions for branch points ([S-V], [Vi]).\\
\\
The considerations of \S \ref{generalite sur les courbures} together with Theorems \ref{thm: integrale de la courbure tangente} and \ref{proposition:curvature of the normal bundle} imply  
\begin{cor}\label{corollaire: inegalite} Under the assumptions of Theorem \ref{proposition:curvature of the normal bundle} 
	\begin{equation}\label{inegalite entre le writhe et le genre}
	|w_\infty(\Sigma)|\leq\sum_{i=1}^k N_i+\chi(\Sigma)-2
	\end{equation}
The equality is attained in (\ref{inegalite entre le writhe et le genre}) if and only if $\Sigma$ is holomorphic for a parallel complex structure on
$\mathbb{R}^4$.
\end{cor}
REMARK. If $\Sigma$ has a single end, it is a braid of algebraic length $e(K)$ and  (\ref{inegalite entre le writhe et le genre}) becomes
\begin{equation}\label{equation:slice Bennequin}
|e(K)|\leq N-1+2g(K)
\end{equation}
 The  inequality (\ref{equation:slice Bennequin}) is just Rudolph's slice-Bennequin inequality ([Ru]).
\section{Knot types at infinity}
In this section, we discuss the braid at infinity of a generic end. To make sense of the word {\it generic}, we will view the space of ends of a given order $N$ as an algebraic variety.
\subsection{Some knot theory reminders}\label{knot theory reminders}
\subsubsection{The $4$-genus of a knot $K$ in $\mathbb{S}^3$} It is the smallest genus of an oriented surface in $\mathbb{B}^4$ bounded by $K$ and we denote it $g_4(K)$. A knot $K$ is {\it ribbon} if it bounds a minimal disk embedded in $\mathbb{B}^4$ ([Ha]). The main known class of examples of ribbon knots are the {\it symmetric unions}, see [La].
\subsubsection{The signature $sign(K)$ of a knot $K$}  For details, see for example [Mu]. Let $S$ be an oriented surface in $\mathbb{S}^3$ bounded by $K$; if $\alpha$ is a loop on $S$, we let $\hat{\alpha}$ be the loop obtained by pushing $\alpha$ slightly in the direction normal to $S$ in $\mathbb{S}^3$. We derive a bilinear form on $H_1(S,\mathbb{Z})$ given by 
\[I(\alpha,\beta)= lk(\hat{\alpha}, \beta) +lk(\hat{\beta}, \alpha).\]
The signature of $I$ is by definition $sign(K)$ and it verifies
\begin{equation}
sign(K)\leq 2g_4(K)
\end{equation}
\subsubsection{The Lissajous toric knots}
These knots (see for example [S-V2]) appear in the study of branch points of minimal surfaces.
Given $N,p,q$ three mutually prime integers, the knot $K(N,q,p)$ is defined in the cylinder $\mathbb{S}^1\times\mathbb{C}$ by
\begin{equation}\label{expression du noeud de lissajous}
e^{i\theta}\mapsto(e^{Ni\theta}, \sin(q\theta), \cos(p\theta+\eta))
\end{equation}
We recall some of their properties
\begin{fact}\label{facts on lissajous}
\begin{enumerate}
	\item 
	Up to mirror symmetry, the phase $\eta$ in (\ref{expression du noeud de lissajous}) does not change the type of the knot 
	\item The knot is ribbon
	\item 
	It is presented as a braid of the form
	\begin{equation}\label{expression generale de la tresse de lissajous}
	Q\prod_{1\leq 2i+1\leq N} \sigma_{2i+1}^{\pm 1}Q^{-1}\prod_{1\leq 2i\leq N} \sigma_{2i}^{\pm 1}
	\end{equation}
	\item 
	If $p$ and $q$ are of different parites, the algebraic length of the braid (\ref{expression generale de la tresse de lissajous}) is $0$.
	
\end{enumerate}
\end{fact}
\subsection{Generic types of ends}
\begin{thm}\label{theoreme sur les bouts}
	The set of ends of a minimal surface of order $N$ is an infinite dimensional real algebraic variety ${\mathcal E}_N$. 
	\begin{enumerate}
		\item 
		Except in a codimension $1$-subvariety ${\mathcal T}_N$ of ${\mathcal E}_N$, the knot at infinity of an end $E\in {\mathcal E}_N$ is a $(N,N-1)$ torus knot and $|w_\infty(E)|=(N-1)^2$.
		\item 
		If $E$ is an end in ${\mathcal T}_N$, except for a codimension $1$ subvariety ${\mathcal S}_N$ of ${\mathcal T}_N$, we have 
		\begin{enumerate}
			\item 
			if $N$ is odd, the knot at infinity is ribbon and $w_\infty(E)=0$
			\item 
			if $N$ is even, $|w_\infty(E)|=1$ and $|sign(K)|\leq N-1$.
		\end{enumerate}
		
	\end{enumerate}
\end{thm}
\subsection{Canonical form of the end}\label{canonical form of the end}
\begin{defn}An end is in a canonical form if it is parametrized by
	\begin{equation}\label{canonical end}
	z\mapsto(z^N+\overline{f(z)}, g(z)+\overline{h(z)})
	\end{equation}
	where $f$, $g$, $h$ are holomorphic functions which are $o(|z|^N)$.
\end{defn}
\begin{prop}
	Any end can be reparametrized as a canonical end.
\end{prop}
\begin{proof}
	The holomorphic part of the first complex coordinate is of the form (cf. (\ref{condition harmonique}))
	\begin{equation}\label{racine Ne de z}
	z^N+e(z)=z^N\Big(1+\frac{e(z)}{z^N}\Big)=\Big[zR(1+\frac{e(z)}{z^N})\Big]^N
	\end{equation}
	where is $e$ holomorphic, $e(z)=o(|z|^N)$ and $R$ is a holomorphic $N$-th root defined in a neighbourhood of $1$. \\
	By the inverse function theorem, we see that $s(z)=zR\Big(1+\frac{e(z)}{z^N}\Big)$ is a diffeomorphism between neighbourhoods $U_1$ and $U_2$ of $\infty$.\\
	Note that $z=s(z)+{\mathcal O}(1)$ and $z^N+e(z)=s(z)^N$.
Thus we rewrite the holomorphic functions $f,g,h$ in terms of $s$, thus putting the end in a canonical form. 
\end{proof}
We define
${\mathcal E}_N$ as the set of all $(f,g,h)$'s where
\begin{itemize}
	\item 
	$f,g,h$ are meromorphic functions whose analytic parts are polynomials of degrees strictly smaller than $N$.
	\item 
	$Nz^{N-1}f'(z)+g'(z)h'(z)=0$
	\item 
	$
	Res(f)=0\ \ \ \ \ Res(g)=\overline{Res(h)}
	\ \ \ \ \ (\star)$	
\end{itemize}
The condition  $(\star)$ enables us to integrate $(z^N+\bar{f}(z), g(z)+\bar{h}(z))$ and the only logarithmic possible terms are of the form $log(z\bar{z})$.
\subsection{The knot at infinity: proof of Theorem \ref{theoreme sur les bouts}}
We consider an end $(f,g,h)$ in ${\mathcal E}_N$.
We reparametrize the end by a $C^1$ diffeomorphism $w$ between two neighbourhoods of infinity such that
\begin{equation}
w(z)^N=z^N+\overline{f(z)}
\end{equation}
Since $Nz^{N-1}f'(z)+g'(z)h'(z)=0$ and $g$ and $h$ are both ${\mathcal O}(|z|^{N-1})$, it follows that  $f$ is ${\mathcal O}(|z|^{N-2})$ and
\begin{equation}
z(w)=w+{\mathcal O}(\frac{1}{|w|})
\end{equation}
So we can write for some $A,B\in\mathbb{C}$,
\begin{equation}\label{ordre de g et h}\left\{
\begin{array}{c}
g'(z)=Az^{N-1}+{\mathcal O}(|z^{N-1}|)=Aw^{N-1}+{\mathcal O}(|w^{N-1}|)\\ h'(z)=Bz^{N-1}+{\mathcal O}(|z^{N-1}|)=
Bw^{N-1}+{\mathcal O}(|w^{N-1}|)
\end{array}\right.
\end{equation} 
{\bf If $\boldsymbol{|A|\neq |B|}$}, the end becomes  after a linear transformation 
\[(Re(w^{N}), Im(w^{N}),\lambda Re(w^{N-1})+o(|w|^{N-1}), \mu Im(w^{N-1})+o(|w|^{N-1}) ) \]
for $\lambda$ and $\mu$ both real and non zero. Thus the knot is the $(N,N-1)$ torus knot. \\
\\
{\bf We now assume that $\boldsymbol{|A|=|B|}$}. Then, for all $w$'s, the $Aw^N+B\bar{w}^N$'s are on the same real line. Indeed, if $A=|A|e^{i\alpha}$, $B=|A|e^{i\beta}$, $w^N=re^{i\theta}$, 
\[Aw^N+B\bar{w}^N=|A|r^N\cos(\theta+\frac{\alpha-\beta}{2}\big)e^{i\big(\frac{\alpha+\beta}{2}\big)}\]
Thus, after a change of coordinates, we rewrite the end as
$w\mapsto$
\begin{equation}
(w^N,Re(w^{N-1})+Re(aw^{N-2})+ o(|w^{N-2}|), Re(bw^{N-2})+Re(cw^{N-3})+ o(|w^{N-3}|))
\end{equation}
for some complex numbers $a$, $b$, $c$.

	 {\bf 1st case. ${\boldsymbol N}$ is odd}\\
	Then $N,N-1$ and $N-2$ are mutually prime. For a generic $b$, the data of the first terms in each coordinate
	\[w\mapsto (w^N,Re(w^{N-1}), Re(bw^{N-2}))\] give an injective map and so they determine the type of the knot. We recognize a Lissajous toric knot $K(N,N-1,N-2)$ of \S \ref{knot theory reminders}: it is ribbon and has writhe $0$. \\
	
	{\bf 2nd case. ${\boldsymbol N}$ is even}\\
	If $sin[(2N-1)\theta]=0$, $(\theta, \theta+\pi)$ is a singular point of the knot $K(N,N-1,N-2)$ given by (\ref{expression du noeud de lissajous}). So to determine the knot at infinity, we introduce the next term in the $4$-th coordinate of the end, i.e.
	\begin{equation}\label{end for an even order1} 
	(w^N,Re(Aw^{N-1})+{\mathcal O}(|w|^{N-2}),
	Re(\lambda w^{N-2}+\mu w^{N-3})+{\mathcal O}(|w|^{N-4})\big)
	\end{equation}The braid corresponding to (\ref{end for an even order1}) is of the form $e^{i\theta}\mapsto$
	\begin{equation}\label{end for an even order2} 
	(R^Ne^{Ni\theta},\underbrace{R^{N-1}\cos(N-1)\theta)}_{\phi_3(\theta)},
	\underbrace{R^{N-2}\cos (N-2)\theta+R^{N-3}\cos((N-3)\theta+\eta)}_{\phi_4(\theta)}\big)
	\end{equation}
	We now use  constructions of [S-V2] to which we refer the reader for details.\\
	A crossing point of the braid is a triple $(t,k,l)$, $t\in[0,1]$, $k,l\in\mathbb{N}$, $0\leq k<l\leq N-1$ such that 
	\begin{equation}\label{equation pour les crossing points}
	\phi_3\big(\frac{2\pi}{N}(t+k)\big)=\phi_3\big(\frac{2\pi}{N}(t+l)\big)
	\end{equation}
	and the sign of this crossing point the sign of
		\begin{equation}
		\label{difference des derivees}
		\underbrace{\Big[\phi_3'\big(\frac{2\pi}{N}(t+k)\big)-\phi_3'\big(\frac{2\pi}{N}(t+l)\big)\Big]}_{(\star)}
		\underbrace{\Big[
		\phi_4\big(\frac{2\pi}{N}(t+k)\big)-\phi_4\big(\frac{2\pi}{N}(t+l)\big)\Big]}_{(\star\star)}
		\end{equation}
The crossing points given by (\ref{equation pour les crossing points}) are of two types
\begin{enumerate}
	\item $|k-l|\neq\frac{N}{2}$
	These are the non singular points of the singular braid defined in $\mathbb{S}^1\times\mathbb{R}^2$ by 
	\begin{equation}
	S: e^{i\theta}\mapsto (e^{Ni\theta},\sin((N-1)\theta),\sin((N-2)\theta))
	\end{equation}
	The orientation reversing map $T:(x,y,z,t)\mapsto (x,y,-z,t)$
	maps a singular (resp. regular) crossing point of $S$ to a 
	singular (resp. regular) crossing point of $S$; moreover $T(S(e^{i\theta}))=T(S(e^{i(\theta+\pi)}))$.\\
	Hence the number of regular points of $S$ counted with sign is $0$.
	\item 
	$|k-l|=\frac{N}{2}$\\
	Similarly to 1. of Fact \ref{facts on lissajous}, the braid is isotopic up to mirror symmetry to
	\[\beta:e^{i\theta}\mapsto\Big(e^{iN\theta}, \sin[(N-1)\theta], \sin[(N-2)\theta+\psi]+\frac{1}{|w|}\sin[(N-3)\theta]\Big)\]
	for some phase $\psi$.\\
	In the factor $(\star\star)$ of (\ref{difference des derivees}) the terms in $\cos (N-2)\theta$ cancel out so the sign of the crossing point is the same as it is for  the braid
	\begin{equation}\label{points singuliers} 
	e^{i\theta}\mapsto \Big(e^{Ni\theta},\cos[(N-1)\theta],
	\cos[(N-3)\theta+\eta]\Big)
	\end{equation}
	for some generic phase $\eta$ ensuring that (\ref{points singuliers}) is non singular. We recognize the Lissajous toric braid $(N,N-1,N-3)$. The crossing points for $|k-l|=\frac{N}{2}$ correspond to the braid generator $\sigma_{\frac{N}{2}}^{\pm 1}$ in the expression (\ref{expression generale de la tresse de lissajous}) of the braid; hence the sum of the exponents of $\sigma_{\frac{N}{2}}$  is $\pm 1$. Hence the writhe of $\beta$ is $\pm 1$.\\
	After a change of phase, $\beta$ is of the form (cf. (\ref{expression generale de la tresse de lissajous}))
	\begin{equation}\label{expression generale de la tresse de lissajous}
		A\prod_{1\leq 2i+1\leq N} \sigma_{2i+1}^{\pm 1}B\prod_{1\leq 2i\leq N} \sigma_{2i}^{\pm 1}
		\end{equation}
	where $B$ differs from $A^{-1}$ only by the sign of the 
	 $\sigma_{\frac{N}{2}}^{\pm}$'s. Since $\sigma_{\frac{N}{2}}^{\pm}$'s appears $\frac{N-1}{2}$ times in $B$, it is enough to change the sign of at most  $\frac{N-1}{2}$ copies of $\sigma_{\frac{N}{2}}^{\pm}$ to transform $B$ into $A^{-1}$. By doing this we transform $\beta$ into a ribbon braid $\hat{\beta}$ which is a symmetric union, hence represents a ribbon knot and has zero signature. We conclude by using the 
	\begin{thm}	(see for example [Cr]) If two knots $K_+$ and $K_-$ have diagrams which differ only by the sign of a single crossing, then
		\[|sign(K_+)-sign(K_-)|\leq 2\]
	\end{thm}
\end{enumerate}
\subsection{When is a PMFTC surface embedded in $\mathbb{R}^4$?}
The knot and of the writhe at infinity can sometimes help determine if a PMFTC surface in $\mathbb{R}^4$ is embedded. \\
If a PMFTC mapping $F$ of $\mathbb{C}$ in $\mathbb{R}^4$ has an end of the form $(z^N+o(|z^N|),z^q+o(|z^q|))$, its knot at infinity is the $(N,q)$ torus knot which has $4$-genus $\frac{(N-1)(q-1)}{2}$ ([K-M]); so $F(\mathbb{C})$ is not embedded. \\
By contrast,  if the knot at infinity is ribbon as in Theorem \ref{theoreme sur les bouts} 2. a), it is not an obstruction for $F(\mathbb{C})$ to be embedded. But of course it is also no guarantee that it will be embedded.\\
See also Proposition \ref{courbure moins six pi} below where  the writhe at infinity is an obstruction for embeddedness in a class of surfaces with two ends. 
\section{Minimal surfaces of small total curvature}
\subsection{Embedded minimal surfaces of total curvature $-4\pi$}
[H-O1]  show that if $F:\Sigma\longrightarrow\mathbb{R}^4$ is minimal of  total curvature is $-4\pi$, $\Sigma$  is either $\mathbb{C}$ or $\mathbb{C}\backslash \{0\}$; in both cases, they give a general formula for the coordinates of $F$ (Propositions 6.4 and 6.6). We take this one step further by investigating when $\Sigma$ is embedded.\\
NB. By {\it holomorphic}, we mean holomorphic for some parallel complex structure $J$ on $\mathbb{R}^4$.
\begin{prop}\label{fourd Enneper}
A non holomorphic minimal embedding $F$ from $\mathbb{C}$ and of total curvature $-4\pi$ can be written as
\begin{equation}\label{formule pour F}
F:z\mapsto(\frac{z^3}{3}-a^2z-\bar{\beta}^2\bar{z},\beta\frac{z^2}{2}+\bar{\beta}\frac{\bar{z}^2}{2}
+\beta a z-\bar{\beta} \bar{a} \bar{z})
\end{equation}
for $a,\beta$ non zero complex numbers with
\begin{equation}\label{condition pour etre injective}
\frac{\bar{\beta}}{\beta}\neq\frac{a^2}{\bar{a}^2}
\end{equation}
The knot at infinity is trivial and $d_+=d_-=1$. 
\end{prop} 
REMARK. If $F$ is of the form (\ref{formule pour F}) without verifying (\ref{condition pour etre injective}), then $F(\mathbb{C})$ has codimension one self-intersections. If we take $a=0$ in formula (\ref{formule pour F}), we have the Enneper surface in $\mathbb{R}^3$ so the surfaces given by (\ref{formule pour F}) can be seen as $4$-dimensional desingularizations of the $3D$ Enneper surface.
\begin{proof}
We look for polynomials $e',f',g',h'$ of degree $0$, $1$ or $2$ which verify $e'f'+g'h'=0$. Without loss of generality we assume 
\begin{equation}\label{formule pour les premieres composantes}
e'=(z-a)(z+a)\ \ \ \ \ f'=\lambda  \ \ \ \ g'=\alpha(z+a)\ \ \ \ h'=\beta(z-a)
\end{equation}
for some $a,\lambda,\alpha,\beta\in\mathbb{C}$.\\
We have $\alpha\beta+\lambda=0$; since the knot at infinity is not a torus knot, this means that  $|\alpha|=|\beta|=\sqrt{|\lambda|}$. We let $\alpha=Re^{i\gamma_1}$, $\beta=Re^{i\gamma_2}$,  multiply the second coordinate in $\mathbb{C}^2$ by $e^{i(\frac{\gamma_2-\gamma_1}{2})}$, assume that 
\begin{equation}\label{conjugues}
\beta=\alpha\ \ \ \ \ \ \ \ \ \ \lambda=-\beta^2
\end{equation}
and derive (\ref{formule pour F}). If $a=0$, the surface is the Enneper surface in $\mathbb{R}^3$ and it has self-intersections. Hence we assume $a\neq 0$.\\ 
We now let $z_1,z_2$ be two different numbers such that 
\begin{equation}\label{egalite pour z1 et z2}
F(z_1)=F(z_2)
\end{equation}
 \[\mbox{We introduce}\ \ \ \ \ \ \ \ X=z_1-z_2\neq 0, \ \ \ Y=z_1+z_2\]
and rewrite the second component in $\mathbb{C}^2$ of (\ref{egalite pour z1 et z2})
\begin{equation}\label{egalite dans la deuxieme composante}
\frac{1}{2}\beta XY+\frac{1}{2}\bar{\beta}\bar{X}\bar{Y}+a\beta X-\bar{a}\bar{\beta}\bar{X}=0
\end{equation}
Separating the real and imaginary part of  of (\ref{egalite dans la deuxieme composante}), we get
\begin{equation}\label{egalite dans la deuxieme composante1}
a\beta X-\bar{a}\bar{\beta}\bar{X}=0\ \ \ \ \ \mbox{and}\ \ \ \ \  \beta XY+\bar{\beta}\bar{X}\bar{Y}=0
\end{equation}
\begin{equation}\label{equation pour Y}
\mbox{hence}\ \ \ \ \ \  X^2=\frac{\bar{a}\bar{\beta}}{a\beta}|X|^2\ \ \ \ \ \ \ \  Y^2=-\frac{a}{\bar{a}}|Y|^2
\end{equation}
We rewrite the first component of $F(z_1)=F(z_2)$ in terms of $X$ and $Y$
\begin{equation}\label{egalite des first component}
\frac{X}{12}(3Y^2+X^2)-a^2X-\bar{\beta}^2\bar{X}=0.
\end{equation}
We plug (\ref{egalite dans la deuxieme composante1}) into (\ref{egalite des first component}), simplify by $X$ and derive
\begin{equation}\label{reecriture du premier terme}
\frac{1}{12}(3Y^2+X^2)-a^2-\frac{a}{\bar{a}}|\beta|^2=0
\end{equation} 
We let $a=|a|e^{iu}$ and rewrite (\ref{reecriture du premier terme}) using (\ref{equation pour Y})
\begin{equation}\label{equation quasi finale}
\frac{1}{12}(-3e^{2iu}|Y|^2+\frac{\bar{a}\bar{\beta}}{a\beta}|X|^2)-|a|^2e^{2iu}-e^{2iu}|\beta|^2=0
\end{equation}
The equation (\ref{equation quasi finale}) has a solution if an only if
\begin{equation}\label{condition pour ne pas etre injective}
\frac{\bar{a}\bar{\beta}}{a\beta}=e^{2iu}=\frac{a}{\bar{a}}
\end{equation}
We recognize  (\ref{condition pour etre injective}).
Conversely, if (\ref{condition pour etre injective}) is verified, we rewrite (\ref{equation quasi finale}) as
\begin{equation}\label{equation finale}
-3|Y|^2+|X|^2=12(|a|^2+|\beta|^2)
\end{equation}
Thus $|X|$ and $|Y|$ belong to a hyperbola ${\mathcal H}$ in $\mathbb{R}^2$: for every point in ${\mathcal H}$, we get four values of the type $(X,Y)$, $(-X,Y)$, $(X,-Y)$ and $(-X,-Y)$. They correspond in turn to two 
double points of $F$ (which are different except if $Y=0$).
\end{proof}

\begin{exple}
The image of the map 
$\mathbb{C}\longrightarrow \mathbb{R}^4$
\[z\mapsto (z+\bar{z}^3,z^2+\frac{3}{4}\bar{z}^2)\]
is an immersed minimal surface of total curvature $-4\pi$ and two transverse double points. It is not holomorphic for any parallel complex structure and its knot at infinity is the $(2,3)$ torus knot.
\end{exple} 
\begin{prop}	
Every non holomorphic minimal immersion of total curvature $-4\pi$ from $\mathbb{C}\backslash\{0\}$ is an embedding. 
Such a map can always be written as
\begin{equation}\label{surface minimale sur le plan epointe}
F:z\mapsto (z+\frac{c^2}{b^2z}+\frac{b^2}{\bar{z}},bln(z\bar{z})+\frac{c}{z}-\frac{\bar{c}}{\bar{z}})
\end{equation}
where $b\in\mathbb{R}\backslash\{0\}, c\in\mathbb{C}$.
\end{prop}
\begin{proof}
	We look for $e',f',g',h'$  which verify $e'f'+g'h'=0$
	and are meromorphic with a single common pole, which we assume to be $0$. After a change of coordinates, we take
	\begin{equation}
	e'=1+\frac{b_1}{z}+\frac{c_1}{z^2},\ \ f'=a+\frac{b_2}{z}+\frac{c_2}{z^2}\ \ g'=\frac{b_3}{z}+\frac{c_3}{z^2},\ \ \
	h'=\frac{b_4}{z}+\frac{c_4}{z^2}
	\end{equation} 
	The residues at $0$ have to verify $Res(f)=0$ and  
	$Res(g)=\overline{Res(h)}$, thus
 \[b_1=\overline{b_2},\ \ \ \ \ \ \ \ \ b_3=\overline{b_4}.\] After a change of coordinates, we assume $b_3$ to be real and we let
\begin{equation}
b_3=b_4=b\in\mathbb{R}
\end{equation}
The coefficients of  $1,\frac{1}{z},\frac{1}{z^2},\frac{1}{z^3},\frac{1}{z^4}$ in the equation $e'f'+g'h'=0$ yield
	\[a=0,\ \ b_2=0=b_1,\ \ c_2+b^2=0,\ \ b(c_3+c_4)=0,\ \  c_1c_2+c_3c_4=0\]
If $b=0$, $c_2=0$ and $c_3$ or $c_4$ is zero, so the surface is complex for some complex structure. Hence $b\neq 0$ and $c_3=-c_4$. Putting
$c=c_4$, we derive the expression (\ref{surface minimale sur le plan epointe}). \\
Now let $z_1$ and $z_2$ be two complex numbers, $z_1\neq z_2$ such that 
\begin{equation}\label{points double avec un trou}
F(z_1)=F(z_2)
\end{equation}
The third coordinate of (\ref{surface minimale sur le plan epointe}) tells us that 
\begin{equation}\label{egalite des modules}
|z_1|=|z_2|
\end{equation}
\[\mbox{We let}\ \ \ \ \ \ c=|c|e^{-i\gamma}, \ \ \ \ w_1=z_1e^{i\gamma}
 \ \ \ \ w_2=z_2e^{i\gamma}\]
so the third component of (\ref{points double avec un trou}) becomes $Im(w_1)=Im(w_2)$.\\
Putting this together with (\ref{egalite des modules}), we derive
\begin{equation}\label{relation entre les w}
w_1=-\bar{w}_2
\end{equation} 
We plug (\ref{relation entre les w}) in the first complex coordinate of (\ref{points double avec un trou}) and derive
\begin{equation}
e^{i\gamma}(w_1+\bar{w}_1)[1+\frac{|c|}{|w|^2}+\frac{b^2}{|w|^2}]=0
\end{equation}
Since  $w_1+\bar{w}_1=w_1-w_2=2e^{i\gamma}(z_1-z_2)\neq 0$, we have a contradiction.
\end{proof}
For larger total curvature, we only get a couple of partial results which we present now.
\subsection{Total curvature $-6\pi$}
\begin{prop}\label{courbure moins six pi}
Let $F:\mathbb{C}\backslash\{0\}\longrightarrow\mathbb{R}^4$ be a minimal surface of total curvature $-6\pi$. If $F$ is embedded and not holomorphic, the two tangent planes at infinity are not transverse.
\end{prop}
\begin{proof}
We have $d_++d_-=3$. Since $F$ is not holomorphic neither of $d_+$ or $d_-$ is zero, thus
\begin{equation}\label{difference entre les d}
|d_+-d_-|=1
\end{equation}
We denote by $K_1$ (resp. $K_2$) the knot at infinity in the neighbourhood of $0$ (resp. infinity). Without loss of generality, we assume that $F$ is equivalent to $z^2$ (resp. $\frac{1}{z}$) near infinity (resp. near $0$). It follows that $K_1$ is trivial and $K_2$ is a knot represented by a braid with $2$ strings. This braid is a $\sigma_1^k$ for some integer $k\neq 0$; if $|k|>1$, then $K_2$ is a torus knot and if $k=\pm 1$, then $K_2$ is trivial. But the knots $K_1$ and $K_2$ are concordant, hence $K_2$ is slice: thus it cannot be a torus knot and $e(K_2)=\pm 1$.\\
If the two tangent plane intersect transversally, Corollary \ref{corollaire: inegalite} yields
\[|d_+-d_-|=|e(K_1)+e(K_2)\pm 4|=|\pm 1 \pm 1\pm 4|\leq 2+1+0-2=1\]
which is impossible.
\end{proof}
\begin{exple}\label{exemple d'une surface non plongee de courbure totale six et a deux bouts}
The following map from $\mathbb{C}\backslash\{0\}$ to $\mathbb{R}^4$ is minimal not embedded
\begin{equation}
z\mapsto (z^2+\ln z+\ln \bar{z},2z-\bar{z}+\frac{1}{2\bar{z}})
\end{equation}
\end{exple}
\subsection{Total curvature $-8\pi$}
\begin{prop}
	The following is a minimal embedding of $\mathbb{C}$ in $\mathbb{R}^4$ of total curvature $-8\pi$. It is preserved by the symmetries w.r.t. the planes generated by $(e_1,e_3)$ and $(e_2,e_4)$.
	\begin{equation}\label{formule pour F1}
	F:z=u+iv\mapsto(\frac{z^5}{5}-z+\bar{z},-\frac{z^3}{3}+\frac{\bar{z}^3}{3}+z+\bar{z})
	\end{equation}
	\[ 	= \left( \begin{array}{c}
	\frac{u^5}{5}-2u^3v^2+uv^4 \\
	\frac{v^5}{5}+u^4v-2u^2v^3-2v\\
	2u\\
	\frac{2v^3}{3}-2u^2v \end{array} \right).\] 
	\end{prop}
\begin{proof}
	The map comes from the $4$-tuple of holomorphic functions
	\[(e',f',g',h')=(z^4-1,1,-z^2+1,z^2+1).\]
Suppose that $F(u_1,v_1)=F(u_2,v_2)$ with $(u_1,v_1)\neq(u_2,v_2)$. We have
\begin{equation}
u_1=u_2=u
\end{equation}
Then $v_1\neq v_2$. Plug this in the $4$-th component and get
\begin{equation}\label{courbure huit1}
v_1^2+v_2^2+v_1v_2=3u^2
\end{equation}
Thus $u\neq 0$, otherwise we would have $v_1=v_2=0$. After simplification, the first component yields
\[
(v_1+v_2)\underbrace{(-2u^2+v_1^2+v_2^2)}_{=\frac{1}{3}(v_1-v_2)^2\ \mbox{using}\ (\ref{courbure huit1})} =0
\]
Thus the double point is of the form $(u,\sqrt{3}u),(u,-\sqrt{3}u)$ 
We plug it into the second component and get $u=0$,  contradiction.
\end{proof}

\footnotesize{Marc Soret: Universit\'e F. Rabelais, D\'ep. de Math\'ematiques, 37000 Tours, France,\\
	Marc.Soret@lmpt.univ-tours.fr\\
\footnotesize{Marina Ville, Univ. Paris Est Creteil, CNRS, LAMA, F-94010 Creteil, France, 
	villemarina@yahoo.fr}
\end{document}